\documentclass{article}

\usepackage{amssymb,amsmath,color}
\usepackage{amsthm}
\usepackage{url}
\usepackage{tikz,subcaption}
\usepackage[enableskew]{youngtab}
\usepackage{ytableau,varwidth
}

\definecolor{red}{rgb}{1,0,0}
\definecolor{blue}{rgb}{.2,.2,.8}

\def\a{\alpha}
\def\b{\beta}
\def\l{\lambda}
\def\m{\mu}

\def\D{\mathcal D}
 
\def\P{\mathcal P} 

\def\Q{\mathcal Q}

\def\vp{\varphi}

\def\mex{\text{mex}}
\def\e{\  }
\def\smex{\sigma \, \mex}

\newtheorem{theorem}{Theorem}[section]
\newtheorem{corollary}[theorem]{Corollary}
\newtheorem{proposition}[theorem]{Proposition}
\newtheorem{conjecture}{Conjecture}

\newtheorem{lemma}[theorem]{Lemma}

\theoremstyle{definition}
\newtheorem{definition}{Definition}
\newtheorem{example}{Example}[section]

\newcommand{\ds}{\displaystyle}

\begin{document}

	\title{Combinatorial Proof of the Minimal Excludant Theorem}
	\author{Cristina Ballantine
	\\
	\footnotesize Department of Mathematics and Computer Science\\
	\footnotesize College of The Holy Cross\\
	\footnotesize Worcester, MA 01610, USA \\
	\footnotesize cballant@holycross.edu
	\and Mircea Merca
	\\ 
	\footnotesize Academy of Romanian Scientists\\
	\footnotesize Ilfov 3, Sector 5, Bucharest, Romania\\
	\footnotesize mircea.merca@profinfo.edu.ro
}
	\date{}
	\maketitle


\begin{abstract}
	
The minimal excludant  of a partition $\l$, $\mex(\l)$, is  the smallest positive integer that is not a part of  $\l$.  For a positive integer $n$,  $ \sigma\,  \mex(n)$ denotes the sum of the minimal excludants of all partitions of $n$. Recently, Andrews and Newman  obtained a new combinatorial interpretations for $\sigma\,  \mex(n)$. They showed, using generating functions,  that $\sigma\, \mex(n)$ equals the number  of partitions of $n$ into distinct parts using two colors. 
 In this paper, we provide a purely combinatorial proof of this result and new properties of the function $\sigma\,  \mex(n)$. We   generalize this combinatorial  interpretation to $\sigma_r\,  \mex(n)$, the sum of least $r$-gaps in all partitions of $n$. The least $r$-gap of a partition $\l$ is the smallest positive integer that does not appear at least $r$ times as a part of $\l$.
\end{abstract}

\section{Introduction}

The minimal excludant or $\mex$-function on a set $S$ of positive integers is the least positive integer not in $S$. The history of this notion goes back to at least the 1930s when it was  applied  to  combinatorial game theory \cite{S35, G39}. 

Recently, Andrews and Newman \cite{AN19} considered the $\mex$-function applied to integer partitions. 
They defined the minimal excludant of a partition $\l$, $\mex(\l)$, to be the smallest positive integer that is not a part of  $\l$. 
In addition, for each positive integer $n$, they defined 
$$ \sigma\,  \mex(n) := \sum_{\l\in \P(n)} \mex(\l),$$
where $\P(n)$ is the set of all partitions of $n$.  Elsewhere in the literature, the minimal excludant  of a partition  $\l$  is referred to as the least gap or smallest gap of $\l$.  An exact and asymptotic formula for $\sigma\,  \mex(n)$ is given in \cite{GK06}.   In \cite{BM19}, where $\mex(\l)$ is denoted by $g_1(\l)$  and $\sigma\,  \mex(n)$ is denoted by $S_1(n)$, the authors study a generalization of $\sigma\,  \mex(n)$ and its connection to polygonal numbers. 

Let $\D_2(n)$ be the set of partitions of $n$ into distinct parts using two colors and let $D_2(n)=|\D_2(n)|$. For ease of notation, we denote the colors of the parts of partitions in $\D_2(n)$ by $0$ and $1$.  In \cite{AN19}, the authors give two proofs of the following theorem. 

\begin{theorem} \label{T1} Given an integer $n \geqslant 0$, we have $$\sigma\, \mathrm{mex}(n)=D_2(n).$$
\end{theorem} They also determine the parity of the $\smex$ function. 

\begin{lemma} \label{L1} $\smex(n)$ is odd if and only if $n=j(3j\pm 1)$ for some non-negative integer $j$. 
\end{lemma} We note that this parity result was also established in \cite[Corollary 1.6]{BM19}. Andrews and Newman write that ``it would be of great interest to have  a bijective proof of Theorem \ref{T1}." They also ask for a combinatorial proof of Lemma \ref{L1}.
In section \ref{combAN}  we provide these desired proofs.  
In the proof of Theorem \ref{T1}, we make use of the fact  that \begin{equation}\label{smex as sum} \sigma \, \mex(n)=\ds\sum_{k\geqslant 0} p(n-k(k+1)/2),\end{equation} where, as usual, $p(n)$ denotes the number of partitions of $n$.
A combinatorial  proof of  \eqref{smex as sum} is given  in \cite[Theorem 1.1]{BM19} .    The same argument is also described in the second proof of \cite[Theorem 1.1]{AN19}. We note the result proven in \cite{BM19} is a  generalization of \eqref{smex as sum} to the sum of $r$-gaps in all partitions of $n$. The $r$-gap of a partition $\l$ is the least positive integer that does not appear $r$ times as a part of $\l$.

In \cite{Andrews12}, Andrews and Merca considered a restricted  $\mex$-function and defined $M_k(n)$ to be the number of partitions of $n$ in which $k$ is the least positive integer that is not a part and there are more parts $>k$ than there are parts $<k$. 

When $k=1$, $M_1(0)=0$ and, if $n>0$, $M_1(n)$ is the number of partitions of $n$ that do not contain $1$ as a part. Thus, if $n>0$, we have $M_1(n)=p(n)-p(n-1)$. Then, from $\eqref{smex as sum}$, we obtain \begin{equation}\label{T2k=1} \sigma\,  \mathrm{mex}(n) -\sigma\,  \mathrm{mex}(n-1) - \delta(n) 
= \sum_{j=0}^\infty M_1\big(n-j(j+1)/2\big),\end{equation} where $\delta$ is the characteristic function of the set of triangular numbers, i.e.,
	$$
	\delta(n) = \begin{cases}
	1,& \text{if $n=m(m+1)/2,\ m\in\mathbb{Z}$,}\\
	0,& \text{otherwise.}
	\end{cases} 
	$$

In section \ref{pfT2}, we prove the following generalization of \eqref{T2k=1}.

\begin{theorem}\label{T2}
	Let $k$ be a positive integer. Given an integer $n\geqslant 0$, we have
	\begin{align*}
	& (-1)^{k-1} \left( \sum_{j=-(k-1)}^k (-1)^j \sigma\, \mathrm{mex}\big(n-j(3j-1)/2\big) - \delta(n) \right) \\
	& = \sum_{j=0}^\infty M_k\big(n-j(j+1)/2\big).
	\end{align*}
\end{theorem}

As a corollary of this theorem we obtain the following infinite family of linear inequalities involving $\sigma\, \mathrm{mex}$. 

\begin{corollary}
	Let $k$ be a positive integer. Given an integer $n\geqslant 0$, we have
	\begin{align*}
		& (-1)^{k-1} \left( \sum_{j=-(k-1)}^k (-1)^j \sigma\, \mathrm{mex}\big(n-j(3j-1)/2\big) - \delta(n) \right) \geqslant 0,
	\end{align*}	
	with strict inequality if $n\geqslant k(3k+1)/2$.
\end{corollary}

In section \ref{pfT2} we also give a combinatorial interpretation for $$\ds  \sum_{j=0}^\infty M_k\big(n-j(j+1)/2\big)$$ in terms of the number of partitions into distinct parts using three colors and satisfying certain conditions. 

In sections \ref{op} and \ref{dist}, we introduce connections to certain subsets of overpartitions and partitions with distinct parts, respectively.

\section{Combinatorial Proofs of Theorem \ref{T1} and Lemma \ref{L1}} \label{combAN}

\subsection{Bijective Proof of Theorem \ref{T1}}\label{s1}

To prove the theorem, we adapt Sylvester's bijective proof of Jacobi's triple product identity \cite{S1882} which was later rediscovered by Wright \cite{W65}. For the interested reader, it is probably easier to follow Wright's short article for the description of the bijection. 

Given a partition $\l$ in $\D_2(n)$, let $\l^{(j)}$, $j=0,1$, be the (uncolored) partition whose parts are the parts of color $j$ in $\l$. Then,  $\l^{(1)}$ and  $\l^{(2)}$ are partitions into distinct parts.

\begin{example} If $\l=4_1+3_0+3_1+2_0+1_0 \in \D_2(13)$, then $\l^{(0)}=3+2+1$ and $\l^{(1)}=4+3$.
\end{example}

Denote by $\eta(k)$  the staircase partition $\eta(k)=k+(k-1)+(k-2)+\cdots +3+2+1$.  If $k=0$ we define $\eta(k)=\emptyset$. 
For any partition $\l$ we denote by $\ell(\l)$ the number of parts in $\l$. 

\begin{definition} Given diagram of left justified rows of boxes (not necessarily the Ferrers diagram of a partition), the \textit{staircase profile} of the diagram is a zig-zag line starting in the upper left corner of the diagram with a right step and continuing in alternating down and right steps until the end of a row of the diagram is reached. 

\end{definition}

\begin{example} The staircase profile of the diagram $$\yng(3,5,8,4,4,2,2,1)$$ is 
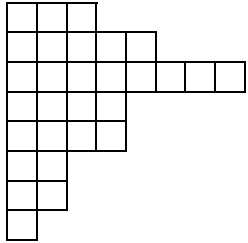
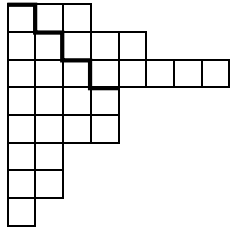
\begin{center}\tiny{\begin{tikzpicture}[inner sep=0in,outer sep=0in]\node (n) {\begin{varwidth}{0cm}{\ydiagram{3,5,8,4,4,2,2,1}
}\end{varwidth}};
\draw[ultra thick] (0,1.47)--(0.37,1.47)-- (0.37,1.1)--(.73,1.1)--(.73,.73)--(1.1,.73)--(1.1,.36)--(1.48,.36) ;
\end{tikzpicture}}\end{center}

\end{example}

Given a Ferrers diagram (with boxes of unit length) of a partition $\l$ into distinct parts, the \textit{shifted Ferrers diagram} of $\l$ is the diagram in which row $i$ is shifted $i-1$ units to the right.

We create a map $$\vp:\bigcup_{k \geqslant 0} \P(n-k(k+1)/2)\to\D_2(n)$$ as follows. 
Start with $\l\in \P(n-k(k+1)/2)$ for some $k\geqslant 0$. Append a diagram with rows of lengths $1, 2, \ldots k$ (i.e., the Ferrers diagram of $\eta(k)$ rotated by $180^\circ$) the top of Ferrers diagram of $\l$. We obtain a diagram with $n$ boxes. Draw the staircase profile of the new diagram. Let $\a$ be the partition whose parts are the length of the columns to the left of the staircase profile and let $\b$ be the partition whose parts are the length of the rows to the right of the staircase profile. Then $\a$ and $\b$ are partitions with distinct parts. Moreover, $k\leqslant \ell(\a)- \ell(\b)\leqslant k+1$. Color the parts of $\a$ with color $k\!\!\pmod 2$ and the parts of $\b$ with color $k+1\!\!\pmod 2$. Then $\vp(\l)$ is defined as the 2-color partition of $n$ whose parts are the colored parts of $\a$ and $\b$.

Conversely, start with  $\mu\in \D_2(n)$. Let $\ell_j(\m)$, $j=0,1$, be the number of parts of color $j$ in $\m$. 


(i) If $\ell_0(\m)\geqslant \ell_1(\m)$, let $r=\ell_0(\m)-\ell_1(\m)$. Let $k=\ds r+\frac{(-1)^r-1}{2}$. Remove the top $k$ rows (i.e., the rotated Ferrers diagram of $\eta(k)$) from the conjugate of the shifted diagram of $\m^{(0)}$ and join the remaining diagram with the shifted digram of $\m^{(1)}$ so they align at the top. The obtained partition $\vp^{-1}(\m)$ belongs to $\P(n-k(k+1)/2)$.

(ii) If $\ell_1(\m)>\ell_0(\m)$, let $r=\ell_1(\m)-\ell_0(\m)$. Let $k=\ds r-\frac{(-1)^r+1}{2}$. Remove the top $k$ rows (i.e., the rotated Ferrers diagram of $\eta(k)$) from the conjugate of the shifted diagram of $\m^{(1)}$ and join the remaining diagram with the shifted digram of $\m^{(0)}$ so they align at the top. The obtained partition $\vp^{-1}(\m)$ belongs to $\P(n-k(k+1)/2)$.

\begin{example} \label{eg1} Let $n=38, k=3$, and let $\l=7+7+6+6+4+2$ be a partition of $n-k(k+1)/2=32$. We add the rotated Ferrers diagram of $\eta(3)$ to the top of the Ferrers diagram of $\l$ and draw the staircase profile. \begin{center}\tiny
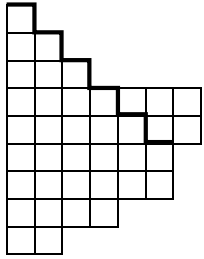
{\begin{tikzpicture}[inner sep=0in,outer sep=0in]\node (n) {\begin{varwidth}{0cm}{\ydiagram{1,2,3,7,7,6,6,4,2}
}\end{varwidth}};
\draw[ultra thick] (-0,1.67)--(0.37,1.67)-- (0.37,1.3)--(.73,1.3)--(.73,.93)--(1.1,.93)--(1.1,.56)--(1.48,.56)--(1.48,.21)--(1.85,.21)--(1.85,-0.16)--(2.22,-0.16) ;
\end{tikzpicture}}\end{center}
Then $\a=9+8+6+5+3+2$ and $\b=3+2$. Since $k$ is odd, we have $\vp(\l)=9_1+8_1+6_1+5_1+3_1+3_0+2_1+2_0\in \D_2(38)$. 

Conversely, suppose $\mu=9_1+8_1+6_1+5_1+3_1+3_0+2_1+2_0\in \D_2(38)$. Then $\ell_0(\mu)=2$ and $\ell_1(\mu)=6$. We have $r=\ell_1(\m)-\ell_0(\m)=4$ and $k=3$. The diagrams of the conjugate of the shifted diagram of $\mu^{(1)}$ and the shifted diagram of $\mu^{(0)}$ are shown below.  
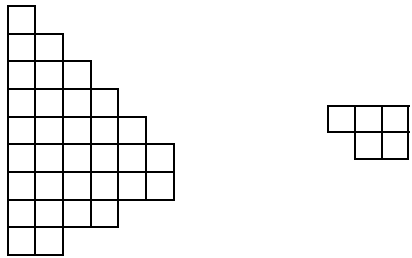
\begin{center} \begin{figure}[htbp]\begin{subfigure}{.7\textwidth}\centering \tiny{\begin{tikzpicture}[inner sep=0in,outer sep=0in]\node (n) {\begin{varwidth}{0cm}{\ydiagram{1,2,3,4,5, 6, 6,  4, 2} 
}\end{varwidth}}; 
\end{tikzpicture}}\end{subfigure}%
\begin{subfigure}{.5\textwidth}
\small{\begin{tikzpicture}[inner sep=0in,outer sep=0in]\node (n) {\begin{varwidth}{0cm}{\young(\e\e\e,:\e\e)
}\end{varwidth}}; 
\end{tikzpicture}}\end{subfigure}%
\end{figure}\end{center}
Next, we remove the first $3$ rows from the conjugate of the shifted diagram of $\mu^{(1)}$ and join the remaining diagram and the shifted digram of $\m^{(0)}$ so they align at the top. We obtain $\vp^{-1}(\mu)=7+7+6+6+4+2\in \P(32)$. 

\end{example}
	
\subsection{Combinatorial Proof of Lemma \ref{L1}}

To determine the parity of $\smex(n)$, we pair partitions in $\D_2(n)$ as follows. If $\m\in \D_2(n)$, we denote by $\tilde \m$ the partition in $\D_2(n)$ obtained by interchanging the colors of the part of $\m$. Then $\m\neq \tilde \m$ if and only if $\m^{(0)}\neq \m^{(1)}$. If $n$ is odd, $\m^{(0)}\neq \m^{(1)}$ for all $\m\in \D_2(n)$ and $\smex$ is even. If $n$ is even, $\smex \equiv q(n/2) \pmod 2$, where, as usual,  $q(m)$ denotes the number of partitions of $m$ with distinct parts. Franklin's involution used to prove Euler's Pentagonal Number Theorem provides a pairing of partitions with distinct parts that shows that $q(m)$ is odd if and only if $m$ is a generalized pentagonal number. Thus, $\smex(n)$ is odd if and only if $n$ is twice a generalized pentagonal number. 

\section{Proofs of Theorem \ref{T2}} \label{pfT2}

\begin{proof}[Analytic proof of Theorem \ref{T2}]

	In \cite{Andrews12}, the authors considered Euler's pentagonal number theorem and proved the following truncated form:
\begin{equation} \label{TPNT}
\frac{(-1)^{k-1}}{(q;q)_\infty} \sum_{n=-(k-1)}^{k} (-1)^{j} q^{n(3n-1)/2}= (-1)^{k-1}+ \sum_{n=k}^\infty \frac{q^{{k\choose 2}+(k+1)n}}{(q;q)_n}
\begin{bmatrix}
n-1\\k-1
\end{bmatrix},
\end{equation}
where
$$(a;q)_n = \begin{cases}
1, & \text{if $n=0$,}\\
\prod\limits_{k=0}^{n-1} (1-aq^k), & \text{otherwise,}
\end{cases}
$$
and
$$
\begin{bmatrix}
n\\k
\end{bmatrix} 
=
\begin{cases}
\dfrac{(q;q)_n}{(q;q)_k(q;q)_{n-k}}, &  \text{if $0\leqslant k\leqslant n$},\\
0, &\text{otherwise.}
\end{cases}
$$
Multiplying both sides of \eqref{TPNT} by
$$
\frac{(q^2;q^2)_\infty}{(q,q^2)_\infty} = \sum_{n=0}^\infty q^{n(n+1)/2},
$$
we obtain
\begin{align*}
& (-1)^{k-1} \left( \bigg( \sum_{n=0}^\infty \sigma\, \mathrm{mex}(n) q^n \bigg) \bigg(  \sum_{n=-(k-1)}^{k} (-1)^{j} q^{n(3n-1)/2}\bigg) -\sum_{n=0}^\infty q^{n(n+1)/2}\right)   \\
& =   \left( \sum_{n=0}^\infty q^{n(n+1)/2} \right) \left( \sum_{n=0}^\infty M_k(n) q^n\right),
\end{align*}
where we have invoked the generating function for $ \sigma\, \mathrm{mex}(n)$ \cite{BM19, AN19},
$$
\sum_{n=0}^\infty \sigma\, \mathrm{mex}(n) q^n = \frac{(q^2;q^2)_\infty}{(q;q)_\infty(q;q^2)_\infty}
$$
and the generating function for $M_k(n)$ \cite{Andrews12},
$$
\sum_{n=0}^\infty M_k(n) q^n = \sum_{n=k}^\infty \frac{q^{{k\choose 2}+(k+1)n}}{(q;q)_n}
\begin{bmatrix}
n-1\\k-1
\end{bmatrix}.
$$
The proof follows easily considering Cauchy's multiplication of two power series.\medskip
\end{proof}

\begin{proof}[Combinatorial proof of Theorem \ref{T2}]

The statement of Theorem \ref{T2} is equivalent to identity \eqref{T2k=1} together with
\begin{align}
& \sigma \, \mex
\left(n-\frac{k(3k+1)}{2}\right)-\sigma \, \mex
\left(n-\frac{k(3k+5)}{2}-1\right)\nonumber \\
& \qquad\qquad = \sum_{j=0}^\infty \Big(M_k\big(n-j(j+1)/2\big) +M_{k+1}\big(n-j(j+1)/2\big)\Big).\label{T2kequiv}
\end{align}
Using \eqref{smex as sum}, identity \eqref{T2kequiv} becomes
\begin{align}
& \sum_{j=0}^\infty \Bigg(p
\bigg(n-\frac{j(j+1)}{2}-\frac{k(3k+1)}{2}\bigg)-p\bigg(n-\frac{j(j+1)}{2}-\frac{k(3k+5)}{2}-1\bigg)\Bigg)\nonumber \\
& \qquad\qquad = \sum_{j=0}^\infty \Big(M_k\big(n-j(j+1)/2\big) +M_{k+1}\big(n-j(j+1)/2\big)\Big). \label{T2kequiv1} 
\end{align}
Identity \eqref{T2kequiv1} was proved combinatorially in \cite{Y15}. Together with the combinatorial proof of \eqref{smex as sum}, this gives a combinatorial proof of Theorem \ref{T2}. 
\end{proof}

\medskip

Next, we  give a combinatorial interpretation for $\ds  \sum_{j=0}^\infty M_k\big(n-j(j+1)/2\big)$. First, we introduce some notation. Given an integer $r$, let $\rm{sign}(r)$ denote the sign of $r$, i.e.
 $$\rm{sign}(r)=\begin{cases} 1 & \mbox{ if } r \geqslant 0\\  -1 & \mbox{ if } r <0.\end{cases}$$
For  integers $k,n$ such that $k\geqslant 1$ and $n\geqslant 0$, we denote by $D^{(k)}_3(n)$  the number of partitions $\mu$ of $n$ into distinct parts using three colors, $0,1,2$, and satisfying the following conditions: 

\begin{enumerate}
	\item[(i)] $\mu$ has exactly $k$ parts of color $2$ and, if $k>1$, twice the smallest part of color $2$ is greater than largest part of color $2$.
	\item[(ii)] Let $r=\ell_0(\mu)-\ell_1(\mu)$ be the signed difference in the number of parts colored $0$ and the number of parts colored $1$ in $\mu$. Let $j=\ds |r|-\frac{1}{2}+\rm{sign}(r)\frac{(-1)^r}{2}$.
	The largest part of color $j \!\!\pmod 2$ must equal $j$ more that the smallest part of color $2$.
\end{enumerate}

%

Then, we have the following proposition.

\begin{proposition}\label{prop}
For  integers $k,n$ such that $k\geqslant 1$ and $n\geqslant 0$, we have \begin{equation}\label{p1}\sum_{j=0}^\infty M_k\big(n-j(j+1)/2\big)=D^{(k)}_3(n).\end{equation}
\end{proposition}

\begin{proof}
Take a partition counted by $M_k\big(n-j(j+1)/2\big)$ and consider its Ferrers diagram. Remove the first $k$ columns and color the length of each of these columns with color $2$. To the remaining Ferrers diagram, add the rotated Ferrers diagram of a staircase $\eta(j)$ of height $j$ and perform the transformation in the combinatorial proof of Theorem \ref{T1}. It is now  straight forward that this transformation is a bijection between the sets of partitions counted by the two sides of \eqref{p1}.\end{proof}

Combining Theorems \ref{T1} and \ref{T2}, and Proposition \ref{prop} we obtain the following corollary which, by the discussion above, has both analytic and combinatorial proofs. 

\begin{corollary}
	For  integers $k,n$ such that $k\geqslant 1$ and $n\geqslant 0$, we have 
	$$(-1)^{\max(0,k-1)} \left( \sum_{j=-\max(0,k-1)}^k (-1)^j \sigma\, \mathrm{mex}\big(n-j(3j-1)/2\big) - \delta(n) \right) \\
	= D^{(k)}_3(n).$$
	
\end{corollary}

Note that, if $k=0$, the statement of the corollary reduces to Theorem \ref{T1}.

\section{Connections with overpartitions} \label{op}

Overpartitions
are ordinary partitions with the added 
condition that the first appearance 
of any part may be overlined or not.
There are eight overpartitions of $3$:
$$
3, \overline{3}, 2+1, \overline{2}+1, 2+\overline{1}, \overline{2}+\overline{1}, 1+1+1, \overline{1}+1+1.
$$
In \cite{AM18}, the authors denoted by  $\overline{M}_k(n)$ the number of overpartitions 
of $n$ in which the first part larger than $k$ appears at least $k+1$ times. 
For example, $\overline{M}_2(12)=16$, 
and the partitions in question are 
$4+4+4$,
$\overline{4}+4+4$,
$3+3+3+3$,
$\overline{3}+3+3+3$,
$3+3+3+2+1$,
$3+3+3+\overline{2}+1$,
$3+3+3+2+\overline{1}$,
$3+3+3+\overline{2}+\overline{1}$,
$\overline{3}+3+3+2+1$,
$\overline{3}+3+3+\overline{2}+1$,
$\overline{3}+3+3+2+\overline{1}$,
$\overline{3}+3+3+\overline{2}+\overline{1}$,
$3+3+3+1+1+1$,
$3+3+3+\overline{1}+1+1$,
$\overline{3}+3+3+1+1+1$,
$\overline{3}+3+3+\overline{1}+1+1$.\medskip

We have the following identity.

\begin{theorem}\label{T3}
	For integers $k,n>0$, we have
	\begin{align*}
	& (-1)^k \left( \sigma\, \mathrm{mex}(n) + 2\sum_{j=1}^{k} (-1)^j \sigma\, \mathrm{mex}(n-j^2) - \delta'(n)\right) \\
	& =\sum_{j=-\infty}^\infty (-1)^j  \overline{M}_k\big(n-j(3j-1)\big),
	\end{align*}
	where
	$$
	\delta'(n) = \begin{cases}
	(-1)^m,& \text{if $n=m(3m-1),\ m\in\mathbb{Z}$,}\\
	0,& \text{otherwise.}
	\end{cases} 
	$$
\end{theorem}

\begin{proof}
	According to \cite[Theorem 7]{AM18}, we have
	\begin{align}
& \frac{(-q;q)_{\infty}} {(q;q)_{\infty}} \left(1 + 2 \sum_{j=1}^{k} (-1)^j q^{j^2} \right) \label{eq:1.11} \\
& \qquad = 1+2 (-1)^k \frac{(-q;q)_k}{(q;q)_k} \sum_{j=0}^{\infty} \frac{q^{(k+1)(k+j+1)}(-q^{k+j+2};q)_{\infty}}{(1-q^{k+j+1})(q^{k+j+2};q)_{\infty}},\notag
\end{align}
where
	$$\sum_{n=0}^\infty \overline{M}_k(n) q^n = 
	2 \frac{(-q;q)_k}{(q;q)_k} \sum_{j=0}^{\infty} \frac{q^{(k+1)(k+j+1)}(-q^{k+j+2};q)_{\infty}}{(1-q^{k+j+1})(q^{k+j+2};q)_{\infty}}.$$
	Multiplying both sides of \eqref{eq:1.11} by
	$$
	(q^2,q^2)_\infty = \sum_{n=-\infty}^\infty (-1)^n q^{n(3n-1)},
	$$
	we obtain
	\begin{align*}
	& (-1)^k \left(\Bigg(\sum_{n=0}^\infty \sigma\, \mathrm{mex}(n) q^n \Bigg) \Bigg( 1 + 2 \sum_{j=1}^{k} (-1)^j q^{j^2}\Bigg) - \sum_{n=-\infty}^\infty (-1)^n q^{n(3n-1)} \right) \\
	& = \left( \sum_{n=-\infty}^\infty (-1)^n q^{n(3n-1)} \right) \left( \sum_{n=0}^\infty \overline{M}_k(n) q^n \right)
	\end{align*}
	and the proof follows easily.
\end{proof}

Related to Theorem \ref{T3}, we remark that
there is a substantial amount of numerical evidence to conjecture the following inequality.

\begin{conjecture}\label{T4}
	For $k,n>0$,
	\begin{align*}
	\sum_{j=-\infty}^\infty (-1)^j  \overline{M}_k\big(n-j(3j-1)\big)\geqslant 0,
	\end{align*}
	with strict inequality if $n\geqslant (k+1)^2$.
\end{conjecture}

It would be very appealing to have a combinatorial interpretation for the sum in this conjecture.

\section{Connections with partitions into distinct parts} \label{dist}

Following the notation for the number of partitions of $n$ into distinct parts of two colors, we denote by $D_1(n)$ the number of partitions of $n$ into distinct parts. We prove the following identity.

%
%

\begin{theorem}\label{T5}
	For any integer $n\geqslant 0$, we have
	\begin{equation}\label{sumdist}
	\sum_{j=0}^\infty (-1)^{j(j+1)/2} \sigma\, \mathrm{mex}\big(n-j(j+1)/2 \big)
	= \sum_{j=0}^\infty D_1\left( \frac{n-j(j+1)/2}{2}\right), 
	\end{equation}
	where $D_1(x)=0$ if $x$ is not a positive integer.
\end{theorem}

\begin{proof}
	Considering the classical theta identity \cite[p. 23, eq. (2.2.13)]{Andrews98}
	\begin{equation}\label{Eq:5}
	\frac{(q^2;q^2)_\infty}{(-q;q^2)_\infty} = \sum_{n=0}^\infty (-q)^{n(n+1)/2},	
	\end{equation}
	we can write
	\begin{align*}
	\left( \sum_{n=0}^\infty \sigma\, \mathrm{mex}(n) q^n \right) \left( \sum_{n=0}^\infty (-q)^{n(n+1)/2} \right) 
	& = \frac{(q^2;q^2)_\infty}{(q;q)_\infty (q;q^2)_\infty} \cdot \frac{(q^2;q^2)_\infty}{(-q;q^2)_\infty} \\
	& = (-q^2;q^2)_\infty \cdot \frac{(q^2;q^2)_\infty}{(q;q^2)_\infty} \\
	& = \left( \sum_{n=0}^\infty D_1(n) q^{2n} \right) \left( \sum_{n=0}^\infty q^{n(n+1)/2} \right) 
	\end{align*}
	and the proof follows by equating the coefficients of $q^n$ in this identity. 
\end{proof}

To obtain a combinatorial interpretation for the sum on the right hand side of \eqref{sumdist}, let $D_2^*(n)$ be the 
 number of partitions of $n$ with distinct parts using two colors such that: (i)  parts of color $0$ form a gap-free partition (staircase) and (ii)
only even parts can have color $1$. 
Then, we have the following identity of Watson type \cite{BM19a}. 

\begin{proposition}
	For $n\geqslant 0$,
	$$\sum_{j=0}^\infty D_1\left( \frac{n-j(j+1)/2}{2}\right)= D_2^*(n).$$	
\end{proposition}

\begin{proof}
To see this, let $\l$ be a partition counted by $\ds D_1\left(\frac{n-j(j+1)/2}{2}\right)$. Double the size of each part of $\l$ to obtain a partition $\mu$ of $n-j(j+1)/2$ whose parts are even and distinct. Color the parts of $\mu$ with color $1$ and add parts $1, 2, \ldots, j$ in color $0$ to obtain a partition counted by $D_2^*(n)$. This transformation is clearly reversible. 	
\end{proof}

In \cite{AM18}, the authors denoted by  $MP_k(n)$ 
the number of partitions of $n$ in which the first part larger than $2k-1$	is odd and appears exactly $k$ times.
All other odd parts appear at most once. For example, $MP_2(19)=10$, and the partitions in question are
$9+9+1$,
$9+5+5$,
$8+5+5+1$,
$7+7+3+2$,
$7+7+2+2+1$,
$7+5+5+2$,
$6+5+5+3$,
$6+5+5+2+1$,
$5+5+3+2+2+2$,
$5+5+2+2+2+2+1$.

We remark the following truncated form of Theorem \ref{T5}.

\begin{theorem}\label{T6}
	For integers $n,k>0$,
	\begin{align*}
	& (-1)^{k-1} \left( \sum_{j=0}^{2k-1} (-1)^{j(j+1)/2} \sigma\, \mathrm{mex}\big(n-j(j+1)/2 \big) -D^*_2(n) \right) \\
	& = \sum_{j=0}^n MP_k(j) D^*_2(n-j).
	\end{align*}
\end{theorem}

\begin{proof}
	According to \cite[Theorem 9]{AM18}, we have
	\begin{align}
	& \frac{(-q;q^2)_\infty}{(q^2;q^2)_\infty} \sum_{j=0}^{2k-1}(-q)^{j(j+1)/2} \label{eq:1.13} \\
	& \qquad = 1 + (-1)^{k-1} \frac{(-q;q^2)_k}{(q^2;q^2)_{k-1}} 
	\sum_{j=0}^{\infty} \frac{q^{k(2j+2k+1)}(-q^{2j+2k+3};q^2)_{\infty}}{(q^{2k+2j+2};q^2)_{\infty}},\notag
	\end{align}
	where
	$$
	\sum_{n=0}^\infty MP_k(n) q^n = \frac{(-q;q^2)_k}{(q^2;q^2)_{k-1}} 
	\sum_{j=0}^{\infty} \frac{q^{k(2j+2k+1)}(-q^{2j+2k+3};q^2)_{\infty}}{(q^{2k+2j+2};q^2)_{\infty}}.
	$$
	The proof follows easily by multiplying both sides of \eqref{eq:1.13} by
	$$\frac{(q^2;q^2)_\infty}{(q;q)_\infty (q;q^2)_\infty} \cdot \frac{(q^2;q^2)_\infty}{(-q;q^2)_\infty}.$$
\end{proof}

A further interesting corollary of Theorem \ref{T6} relates to $\sigma\, \mathrm{mex}(n)$.

\begin{corollary}
		For integers $n,k>0$,
	\begin{align*}
	(-1)^{k-1} \left( \sum_{j=0}^{2k-1} (-1)^{j(j+1)/2} \sigma\, \mathrm{mex}\big(n-j(j+1)/2 \big) -D^*_2(n) \right) \geqslant 0,
	\end{align*}
	with strict inequality if $n\geqslant k(2k+1)$.
\end{corollary}

On the other hand, the reciprocal of the infinite product in \eqref{Eq:5} is the generating function for $\mathrm{pod}(n)$, the number of partitions of $n$ in which odd parts are not repeated, i.e.,
\begin{equation}\label{gfpod}
\frac{(-q;q^2)_\infty}{(q^2;q^2)_\infty} = \sum_{n=0}^\infty \mathrm{pod}(n) q^n.
\end{equation}
The properties of the partition function $\mathrm{pod}(n)$ were studied in \cite{Hirschhorn} by Hirschhorn and Sellers. We easily deduce the following convolution identity.

\begin{corollary}
	For $n\geqslant 0$,
	$$ 
	\sigma\, \mathrm{mex}(n) = \sum_{j=0}^n  \mathrm{pod}(j) D^*_2(n-j).
	$$	
\end{corollary}

Finally, we remark that  finding a combinatorial interpretation for 
$$\sum_{j=0}^n  MP_k(j) D^*_2(n-j) $$
would be very desirable.

\section{Concluding remarks}

The present work began with the search for a combinatorial proof of Theorem \ref{T1}. We were further able to prove several truncated series formulas involving the function $\sigma \, \mex$. In \cite{BM19}, we worked with the generalization of this function: the sum, $S_r(n)$, of $r$-gaps in all partitions of $n$. To keep notation uniform, we use $\smex_r(n)$ for $S_r(n)$. Recall that the $r$-gap of a partition $\l$ is the least positive integer that does not appear at least $r$ times as a part of $\l$. In \cite{BM19}, we proved combinatorially that \begin{equation}\smex_r(n)=\sum_{j \geqslant 0}p(n-rj(j+1)/2), \end{equation} 
and we gave the generating function for $\sigma_r \, \mex(n)$, namely \begin{equation}\label{sr}\sum_{n\geq 0} \sigma_r \, \mex(n)q^n=\frac{(q^{2r};q^{2r})_\infty}{(q;q)_\infty (q^r;q^{2r})_\infty}.\end{equation} 
Denote by $\widetilde D_2^{(r)}(n) $  the number of partitions $\l$ of $n$  using two colors, $0$ and $1$, such that: 
\begin{enumerate}
\item[(i)] $\l^{(0)}$ is a partition into distinct parts divisible by $r$. 
\item[(ii)]  $\l^{(1)}$ is a partition with parts repeated at most $2r-1$ times.
\end{enumerate}

The following generalization of Theorem \ref{T1} is immediate from \eqref{sr}. 

\begin{theorem} \label{TL}
 Let $n,r$ be integers with $r>0$ and $n\geq 0$. Then $\sigma_r \, \mex(n)=\widetilde D_2^{(r)}(n).$ \end{theorem}


\begin{proof}[Combinatorial proof of Theorem \ref{TL}]

Let $\widetilde \D_2^{(r)}(n)$ be the set of partitions of $n$ counted by $\widetilde D_2^{(r)}(n)$ described above. 
Let $\P_{r}(n)$ be the set of partitions of $n$ in which all parts are divisible by $r$. Let $\overline \P_{r}(n)$ be the set of partitions of $n$ in which all parts are not divisible by $r$. Finally, let $\Q_r(n)$ be the set of partitions of $n$ with parts repeated at most $r-1$ times.

Let $\psi$ denote Glaisher's bijection from $\overline \P_r(n)$ to $\Q_r(n)$. 

We create a bijection $$\xi: \bigcup_{j \geqslant 0}\P(n-rj(j+1)/2)\to  \widetilde \D_2^{(r)}(n). $$

Start with a partition $\l \in \P(n-rj(j+1)/2)$ for some $j \geqslant 0$. Let $\tilde \l$ be the partition consisting of the parts of $\l$ that are divisible by $r$ and $\bar \l$ be the partition consisting of the remaining parts of $\l$. Thus all parts of $\bar \l$ are not divisible by $r$. 

Let $\tilde \l_{/r}$ be the partition obtained from $\tilde \l$ by dividing each part by $r$. To $\tilde \l_{/r}\in\P\left(\frac{n-|\bar \l|}{r}-\frac{j(j+1)}{2}\right)$ we apply  the bijection $\vp$ from the combinatorial proof of Theorem \ref{T1} in section \ref{combAN}. (The appended rotated staircase is $\eta(j)$.) Then $\vp( \tilde \l_{/r})\in \D_2\left(\frac{n-|\bar\l|}{r}\right)$.  In $\vp( \tilde \l_{/r})$,  multiply each part of color $0$ by $r$ and repeat each part of color $1$   exactly $r$ times. These parts, together with the parts of $\psi(\bar \l)$ colored $1$, form  the partition $\xi(\l)\in  \widetilde \D_2^{(r)}(n)$.

\bigskip

Conversely, let $\mu\in \widetilde \D_2^{(r)}(n)$. Then $\mu^{(0)}$ is a partition with distinct parts all of which are multiples of $r$ and $\mu^{(1)}$ is a partition with parts repeated at most $2r-1$ times. We write $\mu^{(1)}$ as $\mu^{(1)}=\a^{(1)}\cup \b^{(1)}$, where all parts of $\a^{(1)}$ have multiplicity exactly $r$ and all parts of $\b^{(1)}$ have multiplicity at most $r-1$. Then $\psi^{-1}(\b^{(1)})$ has no part divisible by $r$. We have $|\mu^{(0)}|=rt_1$, $|\a^{(1)}|=rt_2$ and $|\b^{(1)}|=n-rt_1-rt_2$ for some non-negative integers $t_1$ and $t_2$. 

Let $\mu^{(0)}_{/r}$ be the partition with parts colored $0$ obtained from $\mu^{(0)}$ by dividing each part by $r$. Then, $\mu^{(0)}_{/r}$ is a partition with distinct parts colored $0$. Let $\a^{(1)}_{\backslash r}$ be the partition with distinct parts colored $1$ with exactly the same set of parts as $\a^{(1)}$. We then apply $\vp^{-1}$ to  $\mu^{(0)}_{/r}\cup \a^{(1)}_{\backslash r} \in \D_2(t_1+t_2)$ to obtain $\vp^{-1}(\mu^{(0)}_{/r}\cup \a^{(1)}_{\backslash r}) \in \P\left(t_1+t_2-\frac{j(j+1)}{2}\right)$ for some non-negative integer $j$. We multiply each part of $\vp^{-1}(\mu^{(0)}_{/r}\cup \a^{(1)}_{\backslash r})$ by $r$. These parts, together with the parts of $\psi^{-1}(\b^{(1)})$, form the partition $\xi^{-1}(\mu)\in \P(n-rj(j+1)/2)$.

\end{proof}

\begin{example} Let $n=167, r=3, j=3$ and consider $$\l=21+21+19+18+18+12+8+8+8+8+6+1+1 \in \P(167-3\cdot6)=\P(149).$$ Then, $$\tilde \l=21+21+18+18+12+6\in \P_3(96),$$  $$\bar \l=19+8+8+8+8+1+1\in \overline\P_3(53),$$ and $$\tilde \l_{/3}=7+7+6+6+4+2\in \P(32).$$ Applying Glaisher's bijection, we have $\psi(\bar \l)=24+19+8+1+1\in \Q_3(53)$.  From Example \ref{eg1}, we have $\vp(\tilde \l_{/3})=9_1+8_1+6_1+5_1+3_1+3_0+2_1+2_0$. Now, we multiply parts of color $0$ by $3$, repeat each part of color $1$ three times, and include the parts of $\psi(\bar \l)$ with color $1$ to obtain \begin{align*}\xi(\l)=& 24_1+19_1+9_1+9_1+9_1+9_0+8_1+8_1+8_1+8_1+6_1+6_1+6_1+\\ &6_0+5_1+5_1+5_1+ 3_1+3_1+3_1+2_1+2_1+2_1+1_1+1_1\in  \widetilde \D_2^{(3)}(167).\end{align*}

Conversely, let \begin{align*}\mu=& 24_1+19_1+9_1+9_1+9_1+9_0+8_1+8_1+8_1+8_1+6_1+6_1+6_1+\\ &6_0+5_1+5_1+5_1+ 3_1+3_1+3_1+2_1+2_1+2_1+1_1+1_1\in  \widetilde \D_2^{(3)}(167).\end{align*}

Then, we have the following relevant partitions. \begin{align*}\mu^{(0)}=& 9_0+6_0\in \P_3(15)\\   \mu^{(1)}=& 24_1+19_1+9_1+9_1+9_1+8_1+8_1+8_1+8_1+6_1+6_1+6_1+\\ & 5_1+5_1+5_1+ 3_1+3_1+3_1+2_1+2_1+2_1+1_1+1_1\in \Q_6(152) \\\a^{(1)}=& 9_1+9_1+9_1+8_1+8_1+8_1+6_1+6_1+6_1+5_1+5_1+5_1+\\ &  3_1+3_1+3_1+2_1+2_1+2_1\in \P(99)\\ \b^{(1)}=& 24_1+19_1+8_1+1_1+1_1\in \Q_3(53)\\ \psi^{-1}(\b^{(1)})=&19_1+8_1+8_1+8_1+8_1+1_1+1_1\in \overline\P_3(53)\\ \mu^{(0)}_{/3}= & 3_0+2_0\\ \a^{(1)}_{\backslash 3}= & 9_1+8_1+6_1+5_1+3_1+2_1\\ \end{align*} Then $\mu^{(0)}_{/3}\cup\a^{(1)}_{\backslash 3} \in \D_2(38)$ and from Example \ref{eg1}, we have $j=3$ and $$\vp^{-1}(\mu^{(0)}_{/3}\cup\a^{(1)}_{\backslash 3})=7+7+6+6+4+2\in\P(32)=  \P\left(38-\frac{3(3+1)}{2}\right).$$ Now we multiply all parts of $\vp^{-1}(\mu^{(0)}_{/3}\cup\a^{(1)}_{\backslash 3})$ by $3$ and include the  parts of $\psi^{-1}(\b^{(1)})$ with the color removed to obtain $$\xi^{-1}(\mu)=21+21+19+18+18+12+8+8+8+8+6+1+1 \in \P(149)=  \P\left(167-3\cdot 6\right).$$ 

\end{example}

\bigskip


\end{document}